\documentclass[11pt, a4paper]{article}

\usepackage[T1]{fontenc}
\usepackage[utf8]{inputenc}
\usepackage{amsmath,amssymb,amsthm, color,latexsym,epsfig,enumerate,a4, graphicx}
\usepackage{theoremref}
\usepackage{xspace}
\usepackage{bbding}
\usepackage{tikz}
\usetikzlibrary{patterns,positioning,decorations}
\usepackage{comment}
\usepackage{enumitem}
\usepackage{lmodern}
\usepackage{microtype}
\usepackage{a4wide}

   \newcommand{\lab}[1]{\label{#1}}                % hides labels
   \newcommand{\thlab}[1]{\thlabel{#1}} % hides labels in theorems etc.  

\theoremstyle{plain}
\newtheorem{theorem}{Theorem}[section]
\newtheorem{lemma}[theorem]{Lemma}
\newtheorem{claim}[theorem]{Claim}

\newtheorem{observation}[theorem]{Observation}

\newtheorem{conjecture}[theorem]{Conjecture}

\newtheorem{remark}[theorem]{Remark}

\theoremstyle{definition}
\newtheorem{definition}[theorem]{Definition}

\newcommand{\eps}{\ensuremath{\varepsilon}}

\newcommand{\calG}{\ensuremath{\mathcal G}}

\newcommand{\Gnp}{\ensuremath{G(n,p)}}
\newcommand{\Gnq}{\ensuremath{G(n,q)}}

\renewcommand\tilde[1]{\widetilde{#1}}

\newcommand\se{\subseteq}
\newcommand\sm{\setminus}
\renewcommand\to{\rightarrow}

\newcommand{\cF}{{\mathcal F}}
\newcommand{\cH}{{\mathcal H}}
\newcommand{\cG}{{\mathcal G}}
\newcommand{\cS}{{\mathcal S}}
\newcommand{\cP}{{\mathcal P}}

\title{The threshold bias of the clique-factor game}
\author{  
  Anita Liebenau
  \thanks{  
  School of Mathematics and Statistics, UNSW Sydney, NSW 2052, Australia. Email: {\tt a.liebenau@unsw.edu.au.} 
  Supported by the Australian research council (DE170100789 and DP180103684).}
  \and
  Rajko Nenadov 
  \thanks{
  	Department of Mathematics, ETH Zurich, Switzerland. Email: {\tt rajko.nenadov@math.ethz.ch.} Supported in part by SNSF grant 200021-175573.
  }
}

\date{}

\begin{document}
\maketitle

\begin{abstract}
Let $r \ge 4$ be an integer and consider the following game on the complete graph $K_n$ for $n \in r \mathbb{Z}$: Two players, Maker and Breaker, alternately claim previously unclaimed edges of $K_n$ such that in each turn Maker claims one and Breaker claims $b \in \mathbb{N}$ edges.  Maker wins if her graph contains a $K_r$-factor, that is a collection of $n/r$ vertex-disjoint copies of $K_r$, and Breaker wins otherwise. In other words, we consider a $b$-biased $K_r$-factor Maker--Breaker game. We show that the threshold bias for this game is of order $n^{2/(r+2)}$. This makes a step towards determining the threshold bias for making bounded-degree spanning graphs and extends a result of Allen et al.\ who resolved the case $r \in \{3,4\}$ up to a logarithmic factor.
\end{abstract}

\section{Introduction}
In this paper we consider biased positional games played on the edge set of the complete graph $K_n$ in which the winning sets are spanning subgraphs. Biased positional games were introduced by Chv\'atal and Erd\H{o}s~\cite{ce1978} in 1978 and form a central part of positional games, see, for example, the monograph by Beck~\cite{BeckBook}, or~\cite{hkss2014, k2014} for a more recent treatment. 

Let $X$ be a finite set and let ${\cF}\subseteq 2^X$ be a family of subsets. The set $X$ is called the {\em board} and ${\cF}$ is referred to as the family of {\em winning sets}. In the $b$-biased \emph{Maker--Breaker} game $(X,{\cF})$, two players called Maker and Breaker play in rounds. In every round Maker claims one previously unclaimed element of $X$ and Breaker responds by claiming $b$ previously unclaimed elements of $X$. Maker wins if she claims  all elements of some $F\in {\cF}$, otherwise Breaker wins the game. By definition a draw is impossible and thus exactly one player has a winning strategy since Maker--Breaker games are perfect information games. 

A certain class of games that received particular attention are Maker--Breaker games played on the edge set of the complete graph on $n$ vertices, denoted by $K_n$, in which case $X$ is the set of all unordered 2-element subsets of $K_n$, denoted by $\binom{[n]}{2}$. In the {\em connectivity game}, the {\em perfect matching game}, the {\em Hamiltonicity game}, and the {\em triangle game}, for example, the winning sets are the edge sets of all spanning trees, all perfect matchings, all Hamilton cycles, and all copies of $K_3$, respectively. When $n$ is large enough these games are heavily in favour of Maker in the {\em unbiased} version when $b=1$. Chv\'atal and Erd\H{o}s~\cite{ce1978} therefore examined the biased variant for these games. Define the threshold bias $b^*$ of a game~$(X,\cF)$ to be the largest integer $b$ such that Maker wins the $b$-biased Maker--Breaker game~$(X,\cF)$. Note that Maker--Breaker games are bias-monotone, that is Maker wins for every $b\le b^*$ and Breaker wins for every $b>b^*$. 

Chv\'atal and Erd\H{o}s found that the threshold bias $b^*$ is of the order $n\ln n$ for the connectivity, the perfect matching, and the Hamiltonicity game; and of order $\sqrt{n}$ for the triangle game. The order of the threshold bias for an $H$-game, the game in which winning sets correspond to copies of $H$ in $K_n$, was later determined by Bednarska and {\L}uczak~\cite{bl2000} for any fixed graph $H$.
Except for the connectivity game, all the aforementioned games 
can be cast in the following common form. Given a graph $H=H_n$,  what is the threshold bias $b^*$ of the Maker--Breaker game played on $K_n$ in which all winning sets are copies of $H_n$? In the case of Hamiltonicity we simply have $H_n = C_n$, a cycle of length $n$, and in the perfect matching game $H_n$ is a collection of $n/2$ vertex disjoint edges. 

There are choices of $H_n$ for which Maker cannot win even if $b = 1$. A trivial such example is $H_n = K_n$, however even for $H_n$ being a complete graph with only $2 \log n$ vertices Maker cannot win the $H_n$-game \cite[Theorem 6.4]{BeckBook}. It turns out that this can be avoided if we restrict our attention to graphs with maximum degree some constant $\Delta$, and let $n$ be sufficiently large. Furthermore, rather than asking for the threshold bias for a specific $H_n$, we seek a \emph{universal} upper bound: given $\Delta$ and $n$, what is the largest $b_\Delta = b_\Delta(n)$ such that, on the one hand, for every graph $H_n$ with at most $n$ vertices and maximum degree at most $\Delta$ Maker can win a $b$-biased $H_n$-game with $b \le b_\Delta$, and on the other hand there exists at least one such $H_n$ for which Breaker can win with bias $b = b_\Delta$ + 1?

Recently, Allen, B\"ottcher, Kohayakawa, Naves, and Person~\cite{abknp2017} showed that $b_\Delta(n)$ is of order at least $\Omega((n/\log n)^{1/\Delta})$\footnote{All asymptotic statements refer to $n$, the number of vertices, tending to $\infty$.}. The triangle-preventing strategy for Breaker due to Chv\'atal and Erd\H{o}s~\cite{ce1978} shows that this is tight up to a factor of $\sqrt{\log n}$ when $\Delta = 2$. Furthermore, when $\Delta = 3$ the authors of~\cite{abknp2017} show that Breaker can win a \emph{$K_4$-factor} game for some $b = \Omega(n^{1/3})$, which shows (almost) optimality in this case as well. Here the $K_4$-factor and, in general, a $K_r$-factor, corresponds to a graph $H_n$ which consists of $\lfloor n/r \rfloor$ vertex-disjoint copies of $K_r$. However, the authors of~\cite{abknp2017} have expressed a belief that their lower bound of $\Omega((n/\log n)^{1/\Delta})$, in general, is not optimal. We provide evidence for this feeling by determining the order of the threshold bias for the $K_{\Delta + 1}$-factor game for all $\Delta \ge 3$.  For $\Delta = 3$, the threshold bias matches the upper bound in~\cite{abknp2017}, while for $\Delta \ge 4$ the exponent of $n$ of the threshold bias is strictly larger than $1/\Delta$.

\begin{theorem} \label{thm:main}
  For any integer $r \ge 4$ there exist $c, C > 0$ such that the following holds for every $n \in r \mathbb{Z}$.
  \begin{enumerate}[label={(\roman*)}]
    \item\label{main:M} If $b < c n^{2/(r+2)}$ then Maker has a winning strategy in the $b$-biased $K_r$-factor game played on the edge set of $K_n$. 

    \item\label{main:B} If $b > C n^{2/(r+2)}$ then Breaker has a winning strategy in the $b$-biased $K_r$-factor game played on the edge set of $K_n$. 
  \end{enumerate}
\end{theorem}
\begin{remark}
By taking $c$ and $C$ to be sufficiently small and large, respectively, we have that the theorem vacuously holds for all $n < n_0$ for any chosen $n_0$. Therefore, we assume throughout the paper that $n$ is as large as needed for the calculations to be correct.
\end{remark}

For $b \ge C n^{2/(r+2)}$ we show that Breaker has a strategy to `isolate' one particular vertex from being in a copy of $K_r$, which clearly prevents Maker's graphs from containing a $K_r$-factor. Somewhat surprisingly, though not uncommon in extremal and probabilistic combinatorics, this turns out to be Breaker's best strategy: as soon as he cannot achieve this Maker is able to build a $K_r$-factor. 

Theorem \ref{thm:main} suggests the following. 
\begin{conjecture}
For all $\Delta \ge 3$, $b_\Delta = \Theta(n^{2/(\Delta + 3)}).$
\end{conjecture}
In other words, we believe that it is not significantly harder for Maker to build any other graph of maximum degree $\Delta$ than a $K_{\Delta+1}$-factor. 
We take justification for this assumption from two similar settings in extremal graph theory and in random graph theory. The celebrated theorem of Hajnal and Szemer\'edi~\cite{hs1970} states that every graph $G$ of minimum degree at least $(1-1/(\Delta+1))n$ contains a $K_{\Delta + 1}$-factor, and that condition is tight. 
Bollob\'as and Eldridge~\cite{be1978}, and independently Catlin~\cite{c1976}, conjectured that the condition $\delta(G)\ge (1-1/(\Delta + 1))n$ is in fact sufficient to contain every graph $H$ with $n$ vertices and maximum degree $\Delta$. 
A similar assumption is made on the threshold bias $p^*$ for the random graph $G(n,p)$ to contain a certain graph $H_n$. Johansson, Kahn and Vu~\cite{jkv2008} showed that $p^*(n)=(n^{-1}\log^{1/\Delta}n)^{2/(\Delta+1)}$ is a threshold function for $G(n,p)$ to contain a $K_{\Delta + 1}$-factor. It is folklore belief that, for every graph $H_n$ on at most $n$ vertices and of maximum degree $\Delta$, the function  $p^*(n)$ is in fact an upper bound on the threshold functions for $G(n,p)$ to contain  $H_n$, see for example Conjecture~1.3 in~\cite{fln2017}. Supporting evidence towards this conjecture is given by Ferber, Luh and Nguyen~\cite{fln2017} who prove it when $H_n$ is almost-spanning, that is when $H_n$ occupies at most $(1-\eps)n$ vertices.

\medskip
\noindent
{\bf Structure of the paper.} \\
In Section 2 we take a little detour and discuss the \emph{probabilistic intuition}, also called the {\em Erd\H{o}s paradigm}. While this paradigm in its basic form does not apply to the problem we consider here, a variation of it due to Allen et al.~\cite{abknp2017} (Theorem \ref{thm:maker_rg}) turns out to give the correct answer. This result will also serve us to provide further intuition why the threshold bias in Theorem~\ref{thm:main} is of the order $n^{2/(r+2)}$, or, more precisely, why Breaker is not able to isolate a single vertex from being in a copy of $K_r$ for $b < cn^{2/(r+2)}$. In Section 3 we fix notation and state preliminary results. In Section 4, we provide Maker's strategy and prove Theorem~\ref{thm:main}~\ref{main:M}. Section 5 is devoted to Breaker's strategy, i.e.~Theorem~\ref{thm:main}~\ref{main:B}.

\section{Probabilistic Intuition Revised} 
Chv\'atal and Erd\H{o}s~\cite{ce1978} found a surprising connection between biased positional games and random graphs. Replace Maker and Breaker by RandomMaker and RandomBreaker, respectively, who choose their edges uniformly at random from all unclaimed edges. At the end of the game, the graph of RandomMaker has the same distribution as $G(n,m)$, a graph with $m$ edges chosen uniformly at random from all $\binom{n}{2}$ possible edges, where $m$ is roughly $\binom{n}{2}/(b+1)$ (we omit floor and ceiling signs unless crucial). It is well known~\cite{JLRbook} that $G(n,m)$ is (a) connected, (b) has a perfect matching, or (c) has a Hamilton cycle with probability tending to 0 if $m\ll n \ln n$, and with probability tending to 1 if $m\gg n \ln n$. That is, the threshold biases of the random version of the connectivity, the perfect matching, and the Hamiltonicity game are of the order $n\ln n$. The results in~\cite{ce1978} imply that the threshold bias $b^*$ in the game with clever players is of the same order of magnitude for the connectivity, the perfect matching, and the Hamiltonicity game. This phenomenon is often called the {\em random graph intuition}, or the {\em Erd\H{o}s paradigm}. In fact, it turns out that the threshold biases for the random and the clever game are asymptotically equal in the connectivity game~\cite{gs2009} and in the Hamiltonicity game~\cite{k2011}. 

It is one of the central questions in positional games to classify games for which the random graph intuition applies. A game which does very much not obey the random graph intuition is the above-mentioned triangle game or, more generally, an $H$-game for a fixed graph $H$ which contains a cycle. It is well-known that the threshold for the appearance of a triangle in $G(n,m)$ is of the order $\Theta(n)$ (see, e.g.,~\cite{JLRbook}). Chv\'atal and Erd\H{o}s~\cite{ce1978}, however, showed that Breaker can prevent a triangle in Maker's graph when playing with a bias $b =\Theta(\sqrt{n})$. 

It follows from Beck's winning criterion for Breaker~\cite{b1982}, a generalisation of the classical Erd\H{o}s-Selfridge criterion to biased games, that Breaker can {\em always} play at least as good as RandomBreaker against RandomMaker. A result by Chv\'atal and Erd\H{o}s~\cite{ce1978} shows that in some cases Breaker can play in a smarter way than just claiming edges at random. Bednarska and \L{}uczak~\cite{bl2000} verified that this is also the case for any $H$-game. However, the main message of their paper is not that the probabilistic intuition completely fails in these games, but rather that it has to be slightly adjusted. 

As mentioned before, if both players play at random then Maker's graph is distributed as a random graph $G(n, m)$ for $m = \binom{n}{2} / (b + 1)$. If Breaker does not play at random then by Maker still playing uniformly at random from the set of all \emph{available} elements we lose control over the distribution of its graphs. To circumvent this, Bednarska and \L{}uczak~\cite{bl2000} suggested the following strategy for Maker: choose a next element uniformly at random from the set of \emph{all} elements (even those that have been previously claimed) and take it only if it forms a valid move, i.e.\ if it has not been previously claimed. Observe that Maker's graph obtained following this strategy is not a random graph but rather a subgraph obtained from a random graph after deleting a few edges. Thus, even though we might not have a fine control over the actual Maker's graph, knowing that it is a subgraph of a random graph turns out to give sufficient information to win an $H$-game. In particular, they show that when $b$ is not too large, the random graph $G(n,\binom{n}{2}/(b+1))$ is {\em globally robust} with respect to containing a copy of $H$, which in turn implies that Maker has a winning strategy. That is, even after removing any small proportion of the edges the plucked random graph still contains a copy of $H$. For a precise definition of robustness we refer the reader to~\cite{sv2008} where a systematic study of this concept was initiated. 

The next step in explaining a connection between Maker--Breaker games and random graphs was done by Ferber, Krivelevich and Naves~\cite{fkn2015}. While the strategy of playing purely at random works well in the case of $H$-games for graphs $H$ of fixed size, it fails when the winning sets are spanning subgraphs of $K_n$ as Breaker can isolate a vertex before Maker is likely to claim an edge incident to that vertex. To manifest the connection between Maker--Breaker games and random graphs for these spanning-graph games, Ferber, Krivelevich and Naves~\cite{fkn2015} provided a {\em local-resilience} analogue to the theorem in~\cite{bl2000} and showed that in a $b$-biased game played on $K_n$, Maker can claim a subgraph of $G(n,m)$ for $m=\Theta(n^2/b)$ such that each vertex is incident to $\Omega(n/b)$ Maker's edges. Thus, if Maker tries to achieve a graph property $\cP$ that cannot be destroyed by deleting a fixed proportion of edges at each vertex then the strategy in~\cite{fkn2015} yields a winning strategy for Maker. In particular, lower bounds on the threshold bias for several games like the perfect matching game, the connectivity and the Hamiltonicity game could be re-established this way, though with a sub-optimal constant factor. 

However, as the reader could guess, the approach via local resilience does not work for all spanning-structure Maker--Breaker games on $K_n$. The property of containing a $K_3$-factor, for example, is not locally resilient as all triangles in $G(n,m)$ containing a fixed vertex $v$ can be destroyed by removing a vanishing proportion of edges incident to every vertex, see e.g.~\cite{huang2012bandwidth}. For the same reason, the property of containing a $K_r$-factor, $r\ge 4$, is not locally resilient and the approach in~\cite{fkn2015} is not applicable. 
Circumventing the short-coming of the resilience-type approaches, Allen, B\"ottcher, Kohayakawa, Naves, and Person~\cite{abknp2017} finally show that Maker can also assume not only that its graph is a subgraph of a random graph with minimum degree of order $n/b$, but also that the neighbourhood of each vertex has sufficiently many edges. The following theorem makes this precise. For a real $p \in [0, 1]$ and an integer $n$, we write $\Gamma \sim G(n,p)$ if $\Gamma$ is formed by starting with an empty graph on $n$ vertices and adding each possible edge with probability $p$, independently of all other edges. Furthermore, $\Gamma\sim G(n,p)$ satisfies a certain property $\cP$ {\em asymptotically almost surely (a.a.s)} if the probability that $\Gamma$ satisfies $\cP$ tends to 1 as $n\to\infty$.

\begin{theorem} \label{thm:maker_rg}
  For every $n$, $\gamma = \gamma(n) \in (0, 1)$, $p \ge 10^8 \gamma^{-2} n^{-1/2}$, and $b \le 10^{-24} \gamma^6 p^{-1}$ the following holds.  In the $b$-biased Maker-Breaker game played on $K_n$, for any fixed strategy of Breaker, if Maker draws a random graph $\Gamma \sim G(n,p)$ then a.a.s.~$\Gamma$ is such that Maker can claim a spanning subgraph $G$ of $\Gamma$ with $\delta(G) \ge (1 - \gamma)np$ and $e_G(N_\Gamma(v)) \ge (1 - \gamma)p^3 n^2 / 2$ for every $v \in V(\Gamma)$.
\end{theorem}

Using Theorem \ref{thm:maker_rg} in combination with a sparse blow-up lemma from~\cite{abhkp2016}, Allen et al.\ \cite{abknp2017} show that, for some $b =\Omega ((n/\log n)^{1/\Delta})$, these {\em neighbourhood properties} are enough for $G$ to contain all graphs of maximum degree $\Delta$ on at most $n$ vertices.

For which $p$ can we guarantee that the neighbourhood properties given by Theorem \ref{thm:maker_rg} guarantee that every vertex of $G$ is contained in a copy of $K_r$? 
The neighbourhood $N_G(v)$ of a vertex $v$ in $G$ has size roughly $pn$, and the graph induced on $N_G(v)$ is a subgraph of $G(n,p)$ that still contains about $(1-\eps)n^2p^3/2$ edges, i.e.~all but a small proportion of edges of $G(n,p)$ in $N_G(v)$ are also edges of $G$. That is, the subgraph of $G$ induced by $N_G(v)$ has roughly the same distribution as the random graph $G(np,p)$, and for the latter to robustly contain a copy of $K_{r-1}$ it is enough to have $p > C(np)^{-2/r}$ for some constant $C$, which translates to $p>Cn^{-2/(r+2)}$. It turns out that this is the main obstacle for Maker to create a $K_r$-factor.

\section{Preliminaries}

We use standard graph-theoretic notation. All considered graphs are finite and simple. Given a graph $G$, we let $e(G)$ and $v(G)$ denote its number of edges and vertices, respectively. Given a set $X \subseteq V(G)$, let $e_G(X)$ denote the number of edges of $G$ with both endpoints in $X$. Similarly, for disjoint subsets $X, Y \subseteq V(G)$ we let $e_G(X, Y)$ denote the number of edges of $G$ with one endpoint in $X$ and the other in $Y$. Given a vertex $v \in V(G)$, we let $N_G(v)$ denote its neighbourhood, and for a set $X$ let $N_G(X) = \bigcup_{v \in X} N_G(v)$. When $G$ is clear from the context, we omit the subscript. For brevity we also omit floors and ceilings, keeping in mind that all the calculations leave enough margin to accumulate all the rounding errors. 
All asymptotic statements refer to $n$, the number of vertices, tending to $\infty$. 
Following standard asymptotic notation we write in particular, $f\ll g$ when $f/g \to 0$ as $n\to\infty$, and $f\gg g$ if $g\ll f$. 

\subsection{Properties of random graphs}

The following well-known estimates on the likely discrepancy of edges and the concentration of degrees in random graphs follow immediately from Chernoff's inequality and the union bound. 
\begin{lemma} \label{lemma:disc}
Let $p = p(n)$ be such that $n^{-1} \le p \le 0.99$. Then a.a.s.~$\Gamma \sim\Gnp$ satisfies the following properties: 
  \begin{itemize}
    \item For all disjoint subsets $X, Y \subseteq V(\Gamma)$ such that $|X| \le |Y|$ we have
    $$
      e(X, Y) = |X||Y|p \pm O\left( |Y| \sqrt{|X| p \log (n/|Y|)} \right);
    $$

    \item For every subset $X \subseteq V(\Gamma)$ we have
    $$
      e(X) = |X|^2p/2 \pm O\left( |X|\sqrt{|X|p \log (n/|X|)} \right);
    $$

    \item For every vertex $v \in V(\Gamma)$ we have 
    $$
      |N(v)| = np \pm O(\sqrt{np \log n})
    $$
  \end{itemize}
\end{lemma}

In order to state the second result we first need some preparation. Given a graph $G$ and $\eps \in [0,1]$, we say that a pair of disjoint subsets $V_1, V_2 \subseteq V(G)$ forms an \emph{$(\eps)$-regular pair} if for $i=1,2$ and for every $V_i' \subseteq V_i$ of size $|V_i'| \ge \eps |V_i|$ we have
$$
  \left| e(V_1', V_2') - |V_1'||V_2'|p \right| \le \eps |V_1'||V_2'| p,
$$
where $p = e(V_1, V_2)/|V_1||V_2|$. 
Note that Lemma \ref{lemma:disc} implies that a.a.s.~every pair of subsets of $\Gnp$ of size, say, at least $\log n / p$, forms an $(\eps)$-regular pair for every fixed $\eps>0$. 

Let $H$ be a graph with vertex set $\{1, \ldots, k\}$. We denote by $\cG(H, n, m, \eps)$ the collection of all graphs $G$ obtained in the following way: (i) The vertex set of $G$ is a disjoint union $V_1 \cup \ldots \cup V_k$ of sets of size $n$; (ii) For each edge $ij \in E(H)$, we add to $G$ an $(\eps)$-regular bipartite graph with $m$ edges between the pair $(V_i, V_j)$. Let $\cG^*(H, n, m, \eps)$ denote the family of all graphs $G \in \cG(H, n, m, \eps)$ which do not contain a copy of $H$. The following result, originally conjectured by Kohayakawa, \L uczak, and R\"odl \cite{kohayakawa1997onk}, was proven by Balogh, Morris, and Samotij \cite{balogh2015independent} and, independently, Saxton and Thomason \cite{saxton2015hypergraph}.

\begin{theorem} \label{thm:KLR}
  Let $H$ be a fixed graph and $\beta > 0$. Then there exist $C, \eps > 0$ and a positive integer $n_0$ such that 
  $$
    \left| \cG^*(H, n, m, \eps) \right| \le \beta^m \binom{n^2}{m}^{e(H)}
  $$
  for every $n \ge n_0$ and every $m \ge Cn^{2 - 1/m_2(H)}$, where
  $$
    m_2(H) = \max\left\{ \frac{e(H') - 1}{v(H') - 2} \colon H' \subset H, \; v(H) \ge 3  \right\}.
  $$
\end{theorem}

Theorem \ref{thm:KLR} states that a random element from $\cG(H, n, m, \eps)$ is highly unlikely to be $H$-free. Even more, it implies that the random graph $G(n,p)$ is unlikely to contain any graph from $\cG^*(H, \tilde n, m, \eps)$ for appropriate $\tilde n$ and $p$.  This is made precise in the following lemma.

\begin{lemma} \label{lemma:KLR_rg}
Let $H$ be a graph such that $m_2(H) \ge 2$. Then there exist $\eps, B > 0$ such that for $n^{-1/m_2(H)} \le p = p(n) \le \ln^{-2} n$, the graph $\Gamma \sim \Gnp$ a.a.s.~has the property that, for every $\tilde n \ge Bp^{-m_2(H)}$, every $m \ge \tilde n^2 p / 2$ and 
every graph $G'\in\cG(H, \tilde n, m, \eps)$, if $G' \subseteq \Gamma$ then $G'$ contains a copy of $H$. 
\end{lemma}

\begin{proof}
Let $\eps$, $C > 0$ be as given by Theorem~\ref{thm:KLR} for $H$ and $\beta = (1/(2e^2))^{e(H)}$, and set $B = (2C)^{m_2(H)}$. Let $\Gamma\sim\Gnp$. In order to prove the lemma it suffices to show that $\mu$ vanishes, where $\mu$ is the expected number of subgraphs of $\Gamma$ that are isomorphic to an element in $\cG^*(H, \tilde n, m, \eps)$. 
Note that  
 \begin{align*}
    \mu &=  \sum_{\tilde n \ge B p^{-m_2(H)}} \sum_{m \ge \tilde n^2 p/2} \sum_{G\in \cG^*} \Pr(G\se \Gamma),
%    |\cG^*(H, \tilde n, m, \eps)| p^{e(H) m},
\end{align*}     
where $\cG^* =\cG^*(H, \tilde n, m, \eps)$. 
Now 
$$
  \Pr(G\se \Gamma)\le \binom{n}{\tilde n k} (\tilde n k)! p^{e(H) m}
$$ 
where $k=v(H)$ for brevity. 
Furthermore, $m \ge C \tilde n^{2 - 1/m_2(H)}$ follows from $m \ge \tilde n^2 p/2$  and $\tilde n \ge B p^{-m_2(H)}$. Thus we can apply the bound on $|\cG^*(H, \tilde n, m, \eps)|$ given by Theorem \ref{thm:KLR}. 
We therefore have that 
  \begin{align}  \label{aux332}  
    \mu &\le \sum_{\tilde n \ge B p^{-m_2(H)}} \sum_{m \ge \tilde n^2 p/2}
    \beta^m \binom{\tilde n^2}{m}^{e(H)} \; \binom{n}{\tilde n k } (\tilde n k)! p^{e(H) m} \nonumber\\
    &\le \sum_{\tilde n \ge B p^{-m_2(H)}}  \sum_{m \ge \tilde n^2 p/2} \binom{n}{\tilde nk} (\tilde n k)! 
    \beta^m \left( \frac{{\tilde n}^2 e}{m}  \right)^{e(H) m} p^{e(H) m} \nonumber\\
    &\le \sum_{\tilde n \ge B p^{-m_2(H)}} \sum_{m \ge \tilde n^2 p/2}  n^{2 \tilde nk}
    \left( \beta^{1/e(H)} 2 e  \right)^{e(H) m} \nonumber\\
    &\le  \sum_{\tilde n \ge Bp^{-m_2(H)}} \sum_{m \ge \tilde n^2 p / 2}  
    \exp\left(2 k\tilde n \ln n - m \right),
%       = o(1),
  \end{align}
  where the third inequality follows from $m\ge \tilde n^2p/2$, and the last inequality follows from our choice of $\beta$ and $e(H)\ge 1$ since $m_2(H)\ge 2$. Now, for sufficiently large $n$, 
  $$
    2 k \tilde n \ln n - m \le \tilde n (2k \ln n - \tilde n p /2) \le \tilde n (2k \ln n - Bp^{-1}/2) < - \ln^2 n,
  $$
from the lower bound on $m$, the lower bound on $\tilde n$ and $m_2(H)\ge 2$, and from the upper bound on $p$, respectively. 
Thus the final expression in~\eqref{aux332} tends to 0 as $n\to \infty$. The assertion of the lemma follows from Markov's Inequality.
\end{proof}

We remark that the condition $m_2(H) \ge 2$ is purely for convenience, and in fact $m_2(H) > 1$ would work as well (having an impact only on the upper bound on $p$). It should be noted that the previous lemma could also be derived from a result of Conlon, Gowers, Samotij, and Schacht \cite{conlon2014klr}. Finally, to apply the previous result in our proof we make use of the following lemma (see, e.g., \cite[Lemma 4.3]{gerke_steger_2005}).

\begin{lemma} \label{lemma:exact_m_edges}
  Given a positive $\eps < 1/6$, there exists a constant $C$ such that any $(\eps)$-regular graph $B = (V_1 \cup V_2, E)$ contains a $(2\eps)$-regular subgraph $B = (V_1 \cup V_2,E')$ with $|E'| = m$ edges for all $m$ satisfying $C |V(B)| \le m \le |E(B)|$.
\end{lemma}

\subsection{System of disjoint hyperedges}

Given a hypergraph $H$, we denote by $\tau(H)$ the size of a smallest \emph{vertex cover} of $H$, that is the size of a smallest subset $X \subseteq V(H)$ that intersects all the edges of $H$. Note that if $H$ and $H'$ are hypergraphs on the same vertex set then $\tau(H \cup H') \le \tau(H) + \tau (H')$. We make use of the following generalisation of Hall's theorem due to Haxell~\cite{haxell1995condition}. 

\begin{theorem} \label{thm:haxell}
  Let $H_1, \ldots, H_t$ be a family of $r$-uniform hypergraphs on the same vertex set. If for every $I \subseteq [t]$ we have $\tau(\bigcup_{i \in I} H_i) \ge 2 r |I|$ then one can choose a hyperedge $h_i \in E(H_i)$ for each $i \in [t]$ such that $h_i \cap h_j = \emptyset$ for distinct $i, j \in [t]$.
\end{theorem}

The theorem from \cite{haxell1995condition} gives a slightly better bound than $2 r |I|$, however for our purposes this is sufficient. 

\section{Maker's strategy} \label{sec:Maker}

Our proof strategy is to show that Maker can build a graph which has certain properties and then show that these properties imply the existence of a $K_r$-factor. The properties we need are summarised in the following definition.

\begin{definition}
  Given $\alpha, \beta, p \in [0,1]$ and $r \in \mathbb{N}$, we say that a graph $G$ with $n$ vertices is \emph{$(\alpha, \beta, p, r)$-neat} if it has the following properties:
  \begin{enumerate}[label={(P\arabic*)}]
    \item \label{prop:expand}
    For every $v \in V(G)$ we have $|N_G(v)| \ge np/2$ and for all disjoint $X, Y \subseteq V(G)$ of size $|X| \ge \log n/p$ and $|Y| \ge \alpha n$ there exists a vertex $v \in X$ with at least $|Y|p/2$ neighbours in $Y$; 

    \item \label{prop:in_nbr}
    For every $v\in V(G)$ and every subset $X \subseteq N_G(v)$ of size $|X| \ge \alpha np$ the induced subgraph $G[X]$ contains a copy of $K_{r-1}$;

    \item \label{prop:chain}
    For all disjoint subsets $V_1, \ldots, V_{r+1} \subseteq V(G)$ of size $|V_i| \ge n^{1 - \beta}$ each, there exists a copy of $K_{r+1}^{-}$ with one vertex in each $V_i$, where $K_{r+1}^{-}$ is a graph obtained by removing an edge from a complete graph with $r+1$ vertices.
  \end{enumerate}
\end{definition}

The next lemma is the heart of the proof of Theorem \ref{thm:main}~\ref{main:M}. It shows that neat graphs contain $K_r$-factors under certain mild conditions on the parameters. 

\begin{lemma} \label{lemma:neat_factor}
  For any integer $r \ge 4$ and a positive $\beta$ there exists positive $\alpha$ and $n_0 \in \mathbb{N}$ such that any $(\alpha, \beta, p, r)$-neat graph with $n_0 < n \in r \mathbb{Z}$ vertices and $p \ge n^{-1/3}$ contains a $K_r$-factor.
\end{lemma}

The proof of Lemma \ref{lemma:neat_factor} follows an approach from \cite{nenadov18triangle} and we postpone it for the next subsection. We now show how Lemma \ref{lemma:neat_factor} implies the first part of Theorem \ref{thm:main}. 

\begin{proof}[Proof of Theorem \ref{thm:main} (i)]
Let $\alpha > 0$ be as given by Lemma~\ref{lemma:neat_factor} for $\beta = 1/(5 (r+2)^2(r-1))$, let $K=K(\alpha)$ be a sufficiently large integer, and suppose $n \in r \mathbb{Z}$ is sufficiently large. Let $p= K n^{-2/(r+2)}$. We show that Maker can build an $(\alpha, \beta, p, r)$-neat graph in the $b$-biased Maker--Breaker game where $b= cp^{-1}$ for some small constant $c$. Such a graph contains a $K_r$-factor by Lemma~\ref{lemma:neat_factor}.

Maker plays two games in parallel: she plays Game 1 in every odd round and Game 2 in every even round, where Game 1 and Game 2 are defined below. Thus both games can be viewed as $(2b)$-biased Maker--Breaker games and can be played completely independently. For the rest of the argument we assume that Breaker has some fixed strategy, that is, for every disjoint pair $(E_M, E_B)$ of subsets of $E(K_n)$, which represents the current set of Maker's and Breaker's edges, he has some fixed rule what to claim next. If we can show that Maker has a winning strategy against an arbitrary such \emph{rulebook}, then she can win regardless of what Breaker plays. In Game 1, Maker's goal is to build a graph $G_1$ that satisfies \ref{prop:expand} and \ref{prop:in_nbr}, and in Game 2 she builds a graph $G_2$ that satisfies \ref{prop:chain}. Overall, this implies that $G_1 \cup G_2$ is an $(\alpha, \beta, p, r)$-neat graph.

  \paragraph{Game 1.} 
Let $\gamma > 0$ be a constant that we specify later, and let $\Gamma$ be a graph on $n$ vertices that has the following properties. 
  \begin{enumerate}[label={($\Gamma$\arabic*)}]
 
    \item \label{game1:maker} 
    In the $(2b)$-biased Maker--Breaker game on $K_n$, Maker has a strategy to claim a spanning subgraph 
    $G \subseteq \Gamma$ with $\delta(G) \ge (1 - \gamma)np$ and 
    $e_G(N_\Gamma(v)) \ge (1 - \gamma)p^3n^2/2$ for every $v \in V(\Gamma)$.

    \item \label{game1:disc} 
    $\Gamma$ satisfies the assertion of Lemma~\ref{lemma:disc}.

    \item \label{game1:r_1} 
    For every $\tilde n\ge Bp^{-r/2}$, every $m\ge \tilde n^2 p/2$ and every graph $G'\in \cG(K_{r-1},\tilde n, m,\eps)$, 
    if $G'\se \Gamma$ then $G'$ contains $K_{r-1}$ as a subgraph, 
    where $B=B(K_{r-1})$ and $\eps=\eps(K_{r-1})$ are the constants from Lemma~\ref{lemma:KLR_rg} applied to $H=K_{r-1}$.

  \end{enumerate}
We argue briefly that such a graph $\Gamma$ exists. Let $\Gamma \sim \Gnp$. 
Then $\Gamma$ satisfies the assertion of~\ref{game1:maker} a.a.s.~by Theorem~\ref{thm:maker_rg} if we choose $K=K(\gamma)$ large enough and $c=c(\gamma)$ small enough. 
Furthermore, $\Gamma$ satisfies \ref{game1:disc} a.a.s.~by Lemma~\ref{lemma:disc}, and it satisfies~\ref{game1:r_1} by Lemma~\ref{lemma:KLR_rg} applied to $H=K_{r-1}$ where we note that $p= K n^{-2/(r+2)} > n^{-2/r} = n^{-1/m_2(K_{r-1})}$ and $m_2(K_{r-1}) \ge 2$. Therefore, we can choose one particular graph $\Gamma$ which has these properties. 

Let $G \subseteq \Gamma$ be a spanning subgraph guaranteed by \ref{game1:maker}. We show that $G$ satisfies \ref{prop:expand} and \ref{prop:in_nbr}. 

For \ref{prop:expand} note that $G$ can be obtained from $\Gamma$ by removing at most $2 \gamma np$ edges touching each vertex since the maximum degree of $\Gamma$ is at most $(1 + \gamma)np$, by \ref{game1:disc}, and since $\delta(G) \ge (1 - \gamma)np$, by \ref{game1:maker}. 
Furthermore, let $X, Y\se V(G)$ be disjoint subsets of size $|X|\ge \log n/p$ and $|Y|\ge \alpha n$, respectively. Then  $e_\Gamma(X, Y) \ge (1 - \gamma)|X||Y|p$ by \ref{game1:disc}. But then at most $|X| \cdot 2\gamma np$ of those edges are not present in $G$ by the preceding observation. 
By choosing $\gamma < \alpha / 8$, we have 
  $$
    e_G(X, Y) \ge (1 - \gamma)|X||Y|p - |X| \cdot 2 \gamma np > |X||Y|p/2.
  $$
  Therefore there exists a vertex  $v \in X$ with at least $|Y|p/2$ neighbours in $Y$.

For \ref{prop:in_nbr} we show that for every $v\in V(G)$, every subset $X\se N_G(v)$ of size $|X|\ge \alpha np$ hosts a copy of some $G'\in \cG(K_{r-1},\tilde n, m,\eps)$, for suitable $\tilde n$, $m$ and $\eps$, which  contains a copy of $K_{r-1}$ by~\ref{game1:r_1}. Fix $v\in V(G)$ and note that we have $|N_{\Gamma}(v)| = (1\pm \gamma) np\gg \log n /p$ by \ref{game1:disc} and assumption on $p$. Thus,  again by \ref{game1:disc},
$$ e_\Gamma(N_\Gamma(v)) \le (1 + \gamma) |N_\Gamma(v)|^2p/2 \le (1 + \gamma)^3 n^2p^3 / 2 < (1 + 4 \gamma) n^2 p^3 / 2, $$ 
where in the last inequality we assumed that $\gamma$ is sufficiently small. From \ref{game1:maker} we conclude that $G[N_\Gamma(v)]$ is `missing' at most $5 \gamma n^2 p^3 / 2$ edges. For brevity, let us upper bound this by $3 \gamma n^2 p^3$. More precisely, there exists a graph $R_v$ on the vertex set $N_\Gamma(v)$ such that $e(R_v) \le 3 \gamma n^2p^3$ and $G[N_\Gamma(v)] = \Gamma[N_\Gamma(v)] \setminus R_v$. Therefore, for all disjoint $X, X' \subseteq N_\Gamma(v)$ we have 
  $$
     e_\Gamma(X, X') - 3 \gamma n^2 p^3 \le e_G(X, X') \le e_\Gamma(X, X'). 
  $$
Additionally, if $|X|,|X'|\gg \log n /p$ then $e_\Gamma(X, X') = (1\pm \gamma)|X||X'|p$ by~\ref{game1:disc}. Let $\eps' = \eps(K_{r-1})/4$, where $\eps(K_{r-1})$ is given in~\ref{game1:r_1}. 
It follows that for any two disjoint subsets $X,X'\se N_{\Gamma}(v)$ of size at least $\eps'\cdot (\alpha n p/r)$ we have 
  \begin{align*}
    e_G(X, X') &\ge (1 - \gamma)|X||X'| p - 3 \gamma n^2 p^3  
    \ge \left(1- \eps'\right) |X||X'| p
  \end{align*}
 and 
$$  e_G(X, X') \le (1+\gamma) |X||X'| p \le \left(1+\eps'\right) |X||X'| p,$$ 
if we choose $\gamma$ small enough in terms of $\eps'$, $\alpha$ and $r$. 
This implies that any two disjoint subsets $V_1, V_2 \subseteq N_\Gamma(v)$ of size $\alpha np/r$ form a $(2 \eps')$-regular pair.

Let now $X \subseteq N_G(v) \subseteq N_\Gamma(v)$ be of size $\alpha np$.  Arbitrarily choose $r-1$ disjoint subsets $V_1, \ldots, V_{r-1} \subseteq X$ of size $ \tilde n = \alpha np / r$. Note that  
$ \tilde n \ge B p^{-r/2} = B p^{-m_2(K_{r-1})}$, where $B=B(K_{r-1})$ is the constant given by~\ref{game1:r_1}, since $p \ge K n^{-2/(r+2)}$ and $K$ is a sufficiently large constant. As previously observed, every $(V_i, V_j)$ forms a $(2\eps')$-regular pair with $m_{i,j} = (1 \pm \eps') \tilde n^2 p$ edges,
  thus we can apply Lemma \ref{lemma:exact_m_edges} to each pair $(V_i, V_j)$ in order to obtain a subset $E_{ij} \subseteq E_G(V_i, V_j)$ of size exactly 
  $$
    m = \tilde n^2 p /2
  $$
  such that $(V_i, V_j)$ is $(4\eps')$-regular, i.e.~$(\eps)$-regular, with respect to $E_{ij}$. This gives us a graph $G' \in \calG(K_{r-1}, \tilde n, m, \eps)$. As $G'$ is a subgraph of $\Gamma$, from \ref{game1:r_1} we conclude that it contains a copy of $K_{r-1}$. Thus, $G$ satisfies~\ref{prop:in_nbr}.

  \paragraph{Game 2.} The properties \ref{prop:in_nbr} and \ref{prop:chain} are achieved in a fairly similar way. Thus, the analysis of Game 2 follows along the lines of the second part of Game 1. There are some crucial differences in the choice of parameters though. 
Recall that $\beta = 1/(5(r+2)^2(r-1))$ and that for property \ref{prop:chain} we want to find a copy of $K_{r+1}^-$ in certain sets of size $n^{1-\beta}$. Let $H=K_{r+1}^-$, let now $\gamma = n^{-3\beta}$ and let $q \in (0,1)$ satisfy 
 \begin{equation}\label{qBounds}
n^{-\frac{1-\beta}{m_2(H)}}\ll q \ll n^{-\frac{2}{r+2}} \gamma^6.
\end{equation} 
Note that this is possible since $m_2(K_{r+1}^-)= \frac{r+2}{2}-\frac{1}{r-1}$ and by choice of $\beta$. Delicate choices for parameters $\gamma$ and $q$ will become apparent soon. We claim that a random graph $\Gamma \sim \Gnq$ has the following properties with high probability. 
  \begin{enumerate}[label={($\Gamma$\arabic*)}]
    \item \label{game2:maker} In the $(2b)$-biased Maker--Breaker game on $K_n$,  Maker has a strategy to claim a spanning subgraph $G \subseteq \Gamma$ with $\delta(G) \ge (1 - \gamma)nq$ and $e_G(N_\Gamma(v)) \ge (1 - \gamma)q^3n^2/2$ for every $v \in V(\Gamma)$.

    \item \label{game2:disc} $\Gamma$ satisfies the assertion of Lemma \ref{lemma:disc};

    \item \label{game2:r_1} 
      For every $\tilde n\ge Bq^{-m_2(H)}$, every $m\ge \tilde n^2 q/2$ and every graph $G'\in \cG(H,\tilde n, m,\eps)$, 
    if $G'\se \Gamma$ then $G'$ contains $H$ as a subgraph, 
    where now $B=B(H)$ and $\eps=\eps(H)$ are the constants from Lemma~\ref{lemma:KLR_rg} applied to $H=K_{r+1}^-$.
    
  \end{enumerate}
Let us verify that $\Gamma$  indeed has these properties~a.a.s. For~\ref{game2:maker} let us verify that the conditions of 
Theorem~\ref{thm:maker_rg} hold. Firstly, $q\ge 10^{8}\gamma^{-2}n^{-1/2}$ is implied by the lower bound in~\eqref{qBounds} since $r\ge 4$ and since $\beta< 1/50$, say. Secondly, recall that $b= cp^{-1}=O(n^{2/(r+2)})$. This together with the upper bound in~\eqref{qBounds} implies that $\gamma^{6} q^{-1} \gg b$, so that \ref{game2:maker} holds a.a.s.~by Theorem~\ref{thm:maker_rg}.
Just as in Game 1, $\Gamma$ satisfies~\ref{game2:disc}~a.a.s.~by Lemma~\ref{lemma:disc}. 
Finally, note that the lower bound in~\eqref{qBounds} implies in particular that $q\ge n^{-1/m_2(H)}$. Thus~\ref{game2:r_1}
holds a.a.s.~by Lemma~\ref{lemma:KLR_rg} applied to $H=K_{r+1}^-$. 
  
Fix $\Gamma$ with these three properties and let $G \subseteq \Gamma$ be a spanning subgraph satisfying \ref{game2:maker}. Crucially, we have sacrificed the value of $q$, which is now significantly smaller than $n^{-2/(r+2)}$, in order to get a smaller error term $\gamma$. Note that in order to guarantee that $G$ satisfies property \ref{prop:in_nbr} in Game 1 we needed $p = \Omega(n^{-2/(r+2)})$. Here it will turn out that a smaller $p$ (which we denote by $q$) suffices provided $\gamma$ is sufficiently small. We now make this precise.

For $\eps' = \eps(H) /4 >0$ consider disjoint subsets $X, X' \subseteq V(G)$ of size at least $\eps' n^{1 - \beta}$. 
First note that 
  \begin{align*}
  e_{\Gamma} (X,X') = (1\pm \gamma)|X||X'|q, 
  \end{align*}
  by~\ref{game2:disc} since $\gamma^2n^{1-\beta}q\gg \log n$ by choice of $\beta$ being small enough and~\eqref{qBounds}.    
As before, we have that $G$ is obtained from $\Gamma$ by removing at most $2 \gamma nq$ edges touching each vertex, which sums to at most $\gamma n^2 q$ removed edges in total. This, together with the above estimate on  $  e_{\Gamma} (X,X')$, implies that
  \begin{align*}
    (1 + \gamma)|X||X'| q \ge e_G(X, X') &\ge (1 - \gamma)|X||X'|q - \gamma n^2 q \\	&\ge (1 - \eps')|X||X'|q, 
  \end{align*}
  since $\gamma < \eps'/2$, say, and $\gamma n^2 = n^{2-3\beta}\ll \eps' |X||X'|/2$.  Note that it was crucial here that $\gamma \ll n^{- 2\beta}$, that is, $\gamma$ polynomially depends on $n$. Therefore, every pair of disjoint subsets $X, Y \subseteq V(G)$ of size $n^{1 - \beta}$ forms a $(2\eps')$-regular pair. 

 The rest of the argument is the same as in the previous case. Consider some disjoint $V_1, \ldots, V_{r+1} \subseteq V(G)$, each of size $\tilde n = n^{1 - \beta}$. As observed, each pair $(V_i, V_j)$ forms a $(2\eps')$-regular pair with $m_{ij} \ge (1 - \eps')\tilde n^2 q$ edges. By Lemma \ref{lemma:exact_m_edges}, there exists a subset $E_{ij} \subseteq E_G(V_i, V_j)$ of size exactly
  $$
    m = \tilde n^2 q / 2
  $$
  such that $(V_i \cup V_j, E_{ij})$ is $(4\eps')$-regular, i.e.~$(\eps)$-regular. This gives us a subgraph $G' \subseteq G$ which belongs to $\calG(K_{r+1}^-, \tilde n, m, \eps)$. From the lower bound in~\eqref{qBounds} we infer that $\tilde n \ge B(H) q^{-m_2(H)}.$   Thus \ref{game2:r_1} implies that $G'$ contains a copy of $K_{r+1}^-$ with one vertex in each $V_i$. 
\end{proof}

\subsection{Neat graphs contain $K_r$-factors (Lemma \ref{lemma:neat_factor})}

The proof of Lemma \ref{lemma:neat_factor} closely follows ideas from \cite{nenadov18triangle} which are, in turn, based on ideas of Krivelevich \cite{krivelevich1997triangle}. Recall that $K_{r+1}^-$ denotes the graph obtained from $K_{r+1}$ by removing an edge. The main building block in the proof is an \emph{$(r, \ell)$-chain}, the graph obtained by sequentially `gluing' $\ell \ge 0$ copies of $K_{r+1}^-$ on a vertex of degree $r-1$ (see Figure \ref{fig:chain}). We define the $(r, 0)$-chain to be a single vertex. A graph is an \emph{$r$-chain} if it is isomorphic to an $(r, \ell)$-chain, for some integer $\ell \ge 0$.

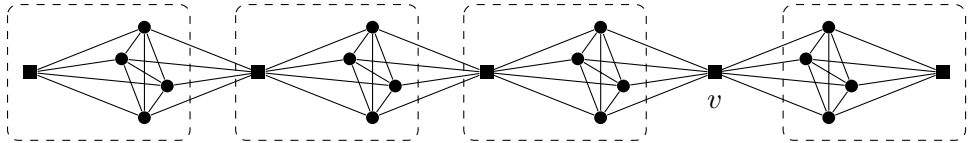
\begin{figure}[h!]   
  \centering
  \begin{tikzpicture}[scale = 0.6]
    \tikzstyle{blob} = [fill=black,circle,inner sep=1.7pt,minimum size=0.5pt]
    \tikzstyle{sq} = [fill=black,rectangle,inner sep=2.5pt,minimum size=2.5pt]

    \foreach \x in {1,...,4}{
      \node[sq] (f\x) at (0 + 5 * \x, 0) {};
      \foreach \c/\y/\z in {1/2.5/1, 2/2.5/-1, 3/2/0.3, 4/3/-0.3}
        \node[blob] (v\x\c) at (0 + 5 * \x + \y, \z) {};       
    }
    \node[sq] (f5) at (25, 0) {};

    \foreach \x/\d in {1/1,2/2,3/3,4/4}{
      \foreach \c in {1,...,4}
        \draw (f\x) -- (v\x\c);

      \draw (v\x1) -- (v\x2) -- (v\x3) -- (v\x4) -- (v\x1);
      \draw (v\x1) -- (v\x3); \draw (v\x2) -- (v\x4);
    }

    \foreach \x/\d in {2/1,3/2,4/3,5/4}{
      \foreach \c in {1,...,4}
        \draw (f\x) -- (v\d\c);      
    }

    \node [below=2pt of f4] {$v$};

    \foreach \x in {1,...,3}
      \draw[dashed, rounded corners] (-0.5 + 5 * \x, 1.5) rectangle (0.5 + 5 * \x + 3, -1.5);

    \draw[dashed, rounded corners] (-0.5 + 5 * 4 + 2, 1.5) rectangle (0.5 + 5 * 5, -1.5);
  \end{tikzpicture}
  \caption{The $(5, 4)$-chain with a $K_5$-factor after removing vertex $v$.} 
  \label{fig:chain}  
\end{figure}

An $(r, \ell)$-chain contains $\ell + 1$ vertices such that removing either of them (but exactly one!) results in a graph which contains a $K_r$-factor. We call such vertices \emph{removable}. If a graph $H$ is an $r$-chain then we use $R(H)$ to denote the set of its removable vertices. We repeatedly use the following observation.

\begin{observation} \label{obs:canonical}
  Let $G$ be a graph and $C_1, \ldots, C_r \subseteq G$ be vertex disjoint $r$-chains. If there exists a copy of $K_r$ in $G$ which intersects each $R(C_i)$ then the subgraph of $G$ induced by $\bigcup_{i \in [r]} V(C_i)$ contains a $K_r$-factor.
\end{observation}

The following lemma together with property \ref{prop:chain} ensures the existence of large $(r, \ell)$-chains. 
\begin{lemma} \label{lemma:long_chain}
Let $G$ be a graph with $n$ vertices such that for every disjoint $X, Y \subseteq V(G)$, each of size at least $\alpha n$, there exists a copy of $K_{r+1}^-$ in $G$ with one vertex of degree $r-1$ in $X$ and all other vertices in $Y$. Then $G$ contains an $(r, \ell$)-chain for every $\ell < (1 - (r+2)\alpha) n/r$. 
\end{lemma}
In the case when $r=3$ this is Lemma~3.1 in~\cite{nenadov18triangle}. Trivial adjustments to that proof give Lemma~\ref{lemma:long_chain}. We omit the proof.

The following, somewhat technical looking lemma provides a crucial \emph{absorbing} property of a collection of $r$-chains that we exploit in the proof of Lemma \ref{lemma:neat_factor}.

\begin{lemma} \label{lemma:absorbing}
  Let $G$ be a graph with $n$ vertices which satisfies \ref{prop:chain} for some $\beta > 0$ and $r$, where $n \ge n_0(\beta, r)$ is sufficiently large. Let $W \subseteq V(G)$ be a subset of size $|W| \ge n/8$, and let $\ell, t \in \mathbb{N}_0$ and $\ell' \ge \ell, t' \in \mathbb{N}$ be such that:
  \begin{itemize}
    \item $(\ell + 1) t' > n^{1 - \beta/2}$, and
    \item $(t + t')(r \ell + 1) < |W|/2$.
  \end{itemize}
  Suppose we are given disjoint $(r, \ell')$-chains $C_1', \ldots, C'_{t'} \subset V(G) \setminus W$. Then there exist disjoint $(r, \ell)$-chains $C_1, \ldots, C_t \subset G[W]$ with the following property: for every $L \subseteq [t]$ there exists $L' \subseteq [t']$ such that the subgraph of $G$ induced by
  $$
    \left( \bigcup_{i \in L} V(C_i) \right) \cup \left( \bigcup_{i \in L'} V(C'_i) \right)
  $$
  contains a $K_r$-factor.
\end{lemma}

Before we prove Lemma \ref{lemma:absorbing} it is instructive to first see how it is used to derive Lemma \ref{lemma:neat_factor}. 

\begin{proof}[Proof of Lemma \ref{lemma:neat_factor}]
  Consider an equipartition $V(G) = V_1 \cup V_2$ chosen uniformly at random. As each vertex has $np/2 \gg \log n$ neighbours (follows from \ref{prop:expand} and the bound on $p$), by Chernoff's inequality and union-bound we have with high probability that every vertex has at least $np/8$ neighbours in $V_1$. Therefore there exists a partition for which this holds. 

  Without loss of generality we may assume that $\beta = 1/k$ for some integer $k \ge 2$. For each $i \in \{1, \ldots, 4k-1\}$ set $\ell_i = n^{1 - (4k-i)/4k}$ and $t_i = n / (32 k (r \ell_i + 1))$. Note that 
  \begin{equation} \label{eq:ell_i_t_i}
    (\ell_i + 1) t_{i+1} = \Theta(n^{1 - \beta/4}).
  \end{equation}

  By repeated application of Lemma \ref{lemma:long_chain} we can find a collection $C_1^{4k-1}, \ldots, C_{t_{4k-1}}^{4k-1} \subseteq G[V_2]$ of pairwise vertex-disjoint $(r, \ell_{4k-1})$-chains. Let us elaborate briefly why this is indeed possible. Such chains occupy $t_{4k - 1} \cdot (r \ell_{4k - 1} + 1) < n / 32$ vertices. Thus if we greedily choose them one by one the set $W \subseteq V_2$ of unoccupied vertices in $V_2$ after every step is of size at least, say, $|W| \ge n/4$. Therefore, by \ref{prop:chain} we have that $G[W]$ satisfies  the assumption of Lemma \ref{lemma:long_chain} for any constant $\alpha > 0$, and consequently it contains an $(r, \ell)$-chain for $\ell < (1 - (r+2)\alpha) |W| / r$. As $\ell_{4k-1} = o(n)$, this proves our claim.

  Let $U_{4k - 1} = \bigcup_{i \in [t_{4k - 1}]} V(C_i^{4k - 1})$. For each $i = 4k - 2, \ldots, 1$, iteratively, let $C_1^i, \ldots, C_{t_i}^i \subset G[V_2] \setminus U_{i+1}$ be disjoint $(r, \ell_i)$-chains given by Lemma \ref{lemma:absorbing} for $C_1^{i+1}, \ldots, C_{t_{i+1}}^{i+1}$ (as $C_1', \ldots, C'_{t'}$) and $W_i = V_2 \setminus U_{i+1}$, and set $U_i = U_{i+1} \cup \bigcup_{j \in [t_i]} V(C_j^i)$. Let us verify that the conditions of Lemma \ref{lemma:absorbing} are met. First, $|W_i| = n/2 - |U_{i+1}|$ and
  $$
    |U_{i+1}| = \sum_{j = i+1}^{4k-1} t_j \cdot (r \ell_j + 1) < n/4.
  $$
  From \eqref{eq:ell_i_t_i} we have $(\ell_i + 1) t_{i+1} > n^{1 - \beta/2}$,  and as
  $$
    (t_i + t_{i+1})(r \ell_i + 1) < 2 t_i (r \ell_i + 1) < n / 16 < |W|/2
  $$
  we can indeed apply Lemma \ref{lemma:absorbing} in each iteration. 

  Finally, let $W_0 = V_2 \setminus U_1$. Apply Lemma \ref{lemma:absorbing} one last time with $\ell_0 = 0$, $t_0 = |W_0|/4$ and $C_1^1, \ldots, C_{t_1}^1$ (as $C_1', \ldots, C_{t'}')$. This is justified as $t_1=\Theta(n^{1-1/4k})$ and $t_0+t_1=o(n)$. The obtained $0$-chains are then just a set of vertices $C_0 \subseteq W_0$ with the property that for every $L_0 \subseteq C_0$ there exists a subset $L_1' \subseteq [t_1]$ such that the subgraph of $G$ induced by
  $$
    L_0 \cup \left( \bigcup_{j \in L_1'} V(C_j^1) \right)
  $$  
  contains a $K_r$-factor. 

  Next, we show that the set $C_0 \cup U_1$ has a strong absorbing property.
  \begin{claim} For any subset $L_0 \subseteq C_0$ such that $|L_0| + |U_1| \in r \mathbb{Z}$, the induced subgraph $G[L_0 \cup U_1]$ contains a $K_r$-factor.
  \end{claim}
  \begin{proof}
  Consider one such $L_0$ and let $L_1' \subseteq [t_1]$ be a subset such that
  $$
    L_0 \cup \left(\bigcup_{j \in L_1'} V(C_j^1) \right)
  $$
  contains a $K_r$-factor. We further take $L_1 = [t_1] \setminus L_1'$ and use the property guaranteed by Lemma \ref{lemma:absorbing} to obtain a subset $L_2' \subseteq [t_2]$ such that the subgraph of $G$ induced by 
  $$
    \left( \bigcup_{j \in L_1} V(C_j^1)  \right) \cup \left( \bigcup_{j \in L_2'} V(C_j^2) \right)
  $$
  contains a $K_r$-factor. Continuing this way, we obtain a subset $L_{4k-1}' \subseteq [t_{4k - 1}]$ such that the subgraph of $G$ induced by 
  $$
    L_0 \cup \bigcup_{i = 1}^{4k - 2} \left( \bigcup_{j \in [t_i]} V(C_j^i) \right)  \cup \left( \bigcup_{j \in L_{4k-1}'} V(C_j^{4k - 1}) \right) = (L_0 \cup U_1) \setminus \bigcup_{j \in L_{4k - 1}} V(C_j^{4k - 1})
  $$
  contains a $K_r$-factor, where $L_{4k - 1} = [t_{4k - 1}] \setminus L_{4k - 1}'$. As $|V(C_j^{4k - 1})| \equiv 1 (\textrm{mod } r)$ and $|L_0| + |U_1| \in r \mathbb{Z}$ we necessarily have $|L_{4k - 1}| \in r \mathbb{Z}$. Therefore, to complete a $K_r$-factor in $G[L_0 \cup U_1]$ it suffices to partition $L_{4k - 1}$ into groups of size $r$ and for each such group $\{i_1, \ldots, i_r\}$ find a copy of $K_r$ with one vertex in each $R(C_{i_1}^{4k - 1}), \ldots, R(C_{i_r}^{4k - 1})$ (see Observation \ref{obs:canonical}). The existence of such $K_r$ follows from \ref{prop:chain} and $|R_j^{4k - 1}| = \ell_{4k - 1} + 1 > n^{1 - 1/4k} > n^{1 - \beta}$. 
  \end{proof}

  We now use this absorbing property to find a $K_r$-factor in $G$. First, let $B \subseteq V_1 \cup (W_0 \setminus C_0)$ be the set of all vertices which are not part of chains and such that they have less than $|C_0|p/2$ neighbours in $C_0$. As $|C_0| \ge \alpha n$, we have $|B| < \log n / p \ll np$, by~\ref{prop:expand} and the lower bound on $p.$ By \ref{prop:in_nbr} and the assumption that every vertex has at least $np/8$ neighbours in $V_1$, we can iteratively take one vertex $v \in B$ at a time and find a copy of $K_r$ which contains $v$ and has all other vertices in $V_1 \setminus B$. This takes care of $B$. Furthermore, we can continue covering the remaining vertices in $V_1 \cup (W_0 \setminus C_0)$ (i.e.\ those which are not part of previously chosen $K_r$'s) with disjoint copies of $K_r$ as long as there are still at least $r n^{1 - \beta}$ vertices, by~\ref{prop:chain}. Let us denote the set of remaining vertices by $L$. With the absorbing property of $C_0 \cup U_1$ in mind, to find a $K_r$-factor of $G$ it now suffices to find vertex-disjoint copies of $K_r$, each of which contains one vertex from $L$ and the others from $C_0$. Whatever we are left with in $C_0$ is guaranteed to form a $K_r$-factor with $U_1$, thus we are done. Note that this is very similar with how we took care of $B$, however the main difference is that $L$ is significantly larger than $B$ and a simple greedy strategy might not work. Instead, we find the desired copies of $K_r$ using Haxell's matching theorem (Theorem~\ref{thm:haxell}). 

  For each $v \in L$ create an $(r-1)$-uniform hypergraph $H_v$ on the vertex set $C_0$ such that $\{v_1, \ldots, v_{r-1}\}$ forms a hyperedge if and only if $\{v, v_1, \ldots, v_{r-1}\}$ form $K_r$ in $G$. If we can find for each $v \in L$ a hyperedge $h_v \in E(H_v)$ such that all these hyperedges are pairwise vertex-disjoint, then we are done.
  To show that such edges exist it suffices to verify Haxell's criterium:
  \begin{equation} \label{eq:verify_haxell}
    \tau(\bigcup_{v \in I} H_v) \ge 2(r-1)|I|
  \end{equation}
  for every $I \subseteq L$. 
  Equivalently, for all subsets $I\se L$ and all $Z \subseteq C_0$ of size $|Z| \le 2(r-1)|I|$  there exists a copy of $K_r$ with one vertex in $I$ and all other vertices in $C_0 \setminus Z$.

  We consider two cases. Consider first the case when $|I| \le \log n /p$ and let $Z$ be some subset of $C_0$ of size at least $2r \log n / p$. As $L \cap B = \emptyset$, every vertex $v \in I$ has at least $|C_0|p/2 > np / 32$ neighbours in $V_1$, thus the subset $X = (N_G(v) \cap C_0) \setminus Z$ is of size at least $np/16$ (we used $np \gg \log n /p$ which follows from the lower bound on $p$). By \ref{prop:in_nbr} there exists a copy of $K_{r-1}$ in $X$. Suppose now that $|L| \ge |I| > \log n /p$ and consider a subset $Z \subseteq C_0$ of size $2r |L| < n/32$. The set $Y = C_0 \setminus Z$ is then of size at least $n/16$. Thus, there exists a vertex $v \in I$ with at least $|Y|p/2 \ge np/32$ neighbours in $Y$,  by \ref{prop:expand}. By \ref{prop:in_nbr} such a neighbourhood contains a copy of $K_{r-1}$, which gives us a desired copy of $K_r$. This finishes the proof.
\end{proof}

It remains to prove Lemma \ref{lemma:absorbing}.

\begin{proof}[Proof of Lemma \ref{lemma:absorbing}]
  By repeated application of Lemma \ref{lemma:long_chain} we can find a collection of $t + t'$ disjoint $(r, \ell)$-chains $C_1, \ldots, C_{t + t'} \in G[W]$. Clearly, for this we could have allowed $W$ to be much smaller than $n/8$, thus this constraint is only for convenience. For each $i \in [t + t']$ we create an auxiliary $(r-1)$-uniform hypergraph $H_i$ on the vertex set $V' = [t']$ by adding a hyperedge $\{j_1, \ldots, j_{r-1}\}$ if and only if there exists a copy of $K_r$ in $G$ with one vertex in each $R(C_i), R(C_{j_1}'), \ldots, R(C_{j_{r-1}}')$. Note that for every such hyperedge the subgraph of $G$ induced by 
  $$
    V(C_i) \cup V(C_{j_1}') \cup \ldots V(C_{j_{r-1}}')
  $$
  contains a $K_r$-factor (see Observation \ref{obs:canonical}).

  We first show that there exists a subset $B \subseteq [t + t']$ of size at most $|B| \le t'$ such that for every subset $J \subseteq [t + t'] \setminus B$ of size $|J| \le t'/8r$ we have
  \begin{equation} \label{eq:2rJ}
    \tau(\bigcup_{i \in J} H_i) \ge 2r|J|.
  \end{equation}
  Initially, set $q = 0$ and $B = \emptyset$. As long as $|B| < t'/8r$ and there exists a subset $J \subseteq [t + t'] \setminus B$ of size $|J| \le t'/8r$ that violates \eqref{eq:2rJ} set $B = B \cup J$, $J_{q+1} = J$ and increase $q$ by 1. Suppose towards a contradiction that for some $q$, $|B| \ge t'/8r$, and let $q$ be the smallest such index. Then $|B| \le t'/4r$ as $|J_q| \le t'/8r$. Moreover, we have
  $$
    \tau(\bigcup_{i \in B} H_i) \le \sum_{j = 1}^q \tau(\bigcup_{i \in J_j} H_i) < \sum_{j = 1}^q 2 r |J_j| = 2r|B| \le t'/2.
  $$
  This implies that there exists a set $\tilde B \se V'$ of size at most $t'/2$ such that every hyperedge $h\in\bigcup_{i\in B}H_i$ intersects $\tilde B$. In other words, there exists $B' \subseteq V'$ of size $|B'| \ge t'/2$ such that there is no copy of $K_r$ in $G$ with one vertex in $\bigcup_{i \in B} R(C_i)$ and the others in each $R(C_{j_2}'), \ldots R(C_{j_{r}}')$ for some distinct $j_2, \ldots, j_{r} \in B'$. Split $B'$ arbitrarily into $r-1$ sets of nearly equal size, denoted by $B_2', \ldots, B_{r}'$, each of size at least $t'/2r$, and set $X_j = \bigcup_{i \in B'_j} R(C_i')$ for $j=2,\ldots,r$. Then each such $X_j$ is of size  at least
  $$
  (\ell' + 1) \frac{t'}{2r} > (\ell + 1) \frac{t'}{2r}.
  $$
  On the other hand, $X_1$ defined as $\bigcup_{i \in B} R(C_i)$ is of size at least 
  $$
    |X_0| \ge (\ell + 1) |B| \ge (\ell + 1) \frac{t'}{8r}.
  $$
  Thus we have $|X_i| \ge n^{1 - \beta}$ by the assumption of the lemma, with room to spare. By \ref{prop:chain} there exists a copy of $K_r$ intersecting each $X_i$, which is a contradiction. Therefore we have that there exists a set  $|B|$ of size less than $t'/8r$ and every subset $J \subseteq [t + t'] \setminus B$ of size $|J| \le t' / 8r$ satisfies \eqref{eq:2rJ}.

  Take an arbitrary $t$-subset $I \subseteq [t + t'] \setminus B$ and relabel $\{C_i\}_{i \in I}$ as $\{C_i\}_{i \in [t]}$. We show that such $(r, \ell)$-chains have the desired property. Consider some $L \subseteq [t]$. First, let $S \subseteq L$ be a smallest subset such that the subgraph of $G$ induced by
  $$
    \bigcup_{i \in L \setminus S} V(C_i)
  $$
  contains a $K_r$-factor. We claim that $|S| < t'/8r$. Suppose towards a contradiction that $|S| \ge t' / 8r$. Consider an equipartition $S = S_1 \cup \ldots \cup S_r$. Then each set $X_i = \bigcup_{j \in S_i} R(C_j)$ is of size 
  $$
    |X_i| \ge (\ell + 1) \frac{t'}{8r^2} > n^{1 - \beta}
  $$
  thus by \ref{prop:chain} there exist a copy of $K_r$ intersecting each $X_i$. Therefore there exists distinct $i_1, \ldots, i_r \in S$ and a copy of $K_r$ intersecting each $R(C_{i_j})$. By Observation \ref{obs:canonical} this is a contradiction with the minimality of $S$. Finally, as $|S| \le t'/8r$ we have that every subset $J \subseteq S$ satisfies \eqref{eq:2rJ} thus we can choose $h_i \in H_i$ for each $i \in S$ such that these edges are pairwise vertex disjoint, by Theorem~\ref{thm:haxell}. Let $L' = \bigcup_{i \in S} h_i$. The construction of such hyperedges implies that the subgraph of $G$ induced 
  $$
    \bigcup_{i \in L} V(C_i) \cup \bigcup_{i \in L'} V(C_i')
  $$
  contains a $K_r$-factor, as desired.
\end{proof}

\section{Breaker's strategy}

The idea behind the proof is that Breaker prevents a fixed vertex $v$ from being in a copy of $K_r$ in Maker's graph. To illustrate why this could be possible, fix a vertex $v\in [n]$ and assume for now that Maker at first only claims edges incident to $v$, as long as there is at least one such unclaimed edge. Breaker responds by claiming $b$ edges incident to $v$ in every round as well, so that at the end of this first stage of the game the set of neighbours of $v$ in Makers graph, denoted by $N_M(v)$, has size roughly $n/b$. For the rest of the game, Breaker only needs to prevent Maker from claiming a copy of $K_{r-1}$ in $N_M(v)$, which is possible if $b\ge C (n/b)^{2/r}$ for some constant $C$ which is independent of $n$, by the result of Bednarska and \L uczak~\cite{bl2000}; or equivalently if $b\ge C n^{2/(r+2)}$ (with a different constant $C$). 

If Maker indeed first claims as many edges incident to $v$ as possible, this would be the end of the proof. Of course, we cannot rely on this assumption. The way to counterfeit it is to divide the attention of Breaker into two: the first $b/2$ claimed edges are incident to $v$, thus preventing its neighbourhood in the Maker's graph from becoming larger than $2n/b$; the second $b/2$ claimed edges lie inside its current neighbourhood and prevent a copy of $K_{r-1}$. Crucially, the board of the game where we want to use the strategy $\cS$ from~\cite{bl2000} will be revealed over time only (as the neighbourhood of $v$ in Maker's graph increases). It turns out that the proof of a static version of the game (where the whole board is `visible') can be turned into a proof of a suitable dynamic version (where the board is revealed over time). Unfortunately, none of the ingredients of the proof is black-boxable so we need to dig into each part.  

Let us introduce necessary notation in a bit more generality than needed for our application. Let $\cH$ be a given hypergraph, say on vertex set $V(\cH)$ and edge set $E(\cH)$, and let $m$ and $b$ be integers. We define the {\em dynamic-board $(\cH,m,b)$-game} as follows. 
Let $V_0 = \emptyset$. The two players Maker and Breaker play in rounds, with Maker going first. For $i\ge 0$, suppose that $i$ rounds have been played and that $V_i\se V(\cH)$ is defined. In round $i+1$, Maker may play either according to {\em Option (a)} in which she claims up to $m$ elements of $V_i$ and sets $V_{i+1}=V_i$ (in case there are less than $m$ elements she claims all of them), or according to {\em Option (b)} in which she chooses elements $v_1,\ldots,v_{\ell}$ (for some $\ell\ge 1$) from $V(\cH) \sm V_i$ and sets $V_{i+1} = V_i\cup\{v_1,\ldots,v_{\ell}\}$ (but Maker does not claim edges in a round when she enlarges the board). In case Option (a) is not possible, Maker is forced to play Option (b), unless it is the end of the game. Afterwards, Breaker claims (up to) $b$ elements in $V_{i+1}$. In case there are less than $b$ unclaimed elements in $V_{i+1}$, Breaker claims all of them. Maker wins if at the end of the game she has claimed all elements of some hyperedge $H\in E(\cH)$. Otherwise, Breaker wins. 

Given a (fixed) graph $H$ and a complete graph $K_n$, we define a {\em dynamic $b$-biased $H$-game} as the $(\cH, 1, b)$-game where the vertex set of $\cH$ are the edges of $K_n$, and the hyperedges of $\cH$ correspond to edge sets of $K_n$ which form a copy of $H$. The following theorem is a generalisation of the mentioned result by Bednarska and \L uczak \cite{bl2000} to the dynamic setting. For the definition of $m_2(H)$, see Theorem \ref{thm:KLR}.

\begin{theorem}\thlab{ourDynamicBreaker}
For every graph $H$ which contains at least three non-isolated vertices there exists a constant $C>0$ such that Breaker has a winning strategy in the dynamic $b$-biased $H$-game played on $K_n$ if $b\ge Cn^{1/m_2(H)}$.
\end{theorem}

We may take $C$ sufficiently large such that the theorem statement is true for small $n$, so that in the proof we can safely assume that $n$ is as large as needed. The proof of Theorem~\ref{ourDynamicBreaker} proceeds along the lines of~\cite{bl2000}. We sketch the argument in the next section, leaving out calculations that are identical to those in~\cite{bl2000}. 

Theorem \ref{ourDynamicBreaker} is all we need to describe Breaker's strategy for isolating a vertex $v$ from being in a copy of $K_r$. 

\begin{proof}[Proof of Theorem \ref{thm:main} (ii)]
Let $r\ge 4$, let $C$ be a large enough constant, let $n$ be an integer and let $b\ge Cn^{2/(r+2)}$. Let $v$ be a fixed vertex of $K_n$. We show that Breaker has a strategy in the $b$-biased Maker--Breaker game played on the edge set of $K_n$ to prevent Maker from claiming a copy of $K_r$ that contains the vertex $v$. Consequently, Maker's graph does not contain a $K_r$-factor. Before we present the strategy of Breaker we describe an auxiliary game that Breaker simulates in parallel. 

Let $T$ be a set of size $2n/b$, disjoint from $V(K_n)$. By \thref{ourDynamicBreaker}, if $b/2\ge C'(2n/b)^{1/m(K_{r-1})}$ then Breaker has a winning strategy $\cS$ in the dynamic $(b/2)$-biased $K_{r-1}$-game played on $K_T$, the complete graph on the vertex set $T$. Equivalently, $b\ge Cn^{2/(r+2)}$ for suitable $C$. 

We now describe the strategy of Breaker in the $b$-biased Maker--Breaker game played on the edge set of $K_n$. Suppose that $i \ge 0$ rounds have been played already. Let $M$ and $B$ denote the graphs formed by Maker's edges and by Breaker's edges, respectively (we suppress dependence on $i$ for clarity of presentation). Breaker maintains the property that every vertex $w\in N_M(v)$ has a (unique) corresponding vertex $t_w\in T$ such that an edge $uw$ in $N_M(v)$ belongs to Maker's (Breaker's) graph if and only if $t_ut_w$ belongs to Maker's (Breaker's) graph in the auxiliary dynamic $b$-biased $K_{r-1}$-game played on $K_T$. Clearly, this is the case before the first round of the game, and we show that Breaker can maintain such a correspondence throughout the game. Set $T_{i} = \{t_w:w\in N_M(v)\}$.  

Let $xy$ denote the edge that Maker claims in round $i+1$. Then Breaker claims up to $b/2$ edges incident to $v$ including $xv$ or $yv$ if those edges are not claimed yet by either of the players. If Breaker has claimed $b' < b/2$ edges and there are no more unclaimed edges incident to $v$, then he claims $b/2 - b'$ arbitrary edges (note that additional edges do not hurt Breaker). For the remaining $b/2$ edges in round $i+1$ we distinguish between three cases (where the latter two are similar). In Case~1, assume that $x\not\in N_M(v) \cup \{v\}$ or $y\not\in N_M(v) \cup \{v\}$. Then Breaker claims $b/2$ arbitrary edges. In Case~2.1, assume that $x=v$ (the case $y=v$ is analogous). Let $t\in T\sm T_{i}$ and set $t_y= t$. In the auxiliary dynamic $K_{r-1}$-game, Breaker pretends that (the auxiliary) Maker plays according to Option (b) and adds the elements $\{t_y  t_w: w\in N_M(v)\}$ to the board (recall that the vertices in the hypergraph corresponding to that game are the edges of $K_T$). In Case~2.2, assume that $x,y\in N_M(v)$. Then Breaker pretends that in the auxiliary dynamic $K_{r-1}$-game Maker plays according to Option (a) and claims the edge $t_xt_y$. In either of Case~2.1 or~2.2, the strategy $\cS$ in the auxiliary game gives $b/2$ edges $e_1,\ldots, e_{b/2} \in E(K_T)$ for Breaker to claim in the auxiliary board. Let $f_1,\ldots,f_{b/2}$ be the corresponding edges in $N_M(v)$, that is $f_i$ is the edge with endpoints $w_i$ and $u_i$ such that $e_i$ has endpoints $t_{w_i}$ and $t_{u_i}$. Breaker then claims $e_1,\ldots, e_{b/2}$ in the auxiliary game and $f_1,\ldots,f_{b/2}$ in the real game. 

We claim that this is indeed a winning strategy. First note that $|N_M(v)|\le 2n/b$ since Breaker claims $b/2$ of the $n-1$ total edges incident to $v$ in every round. Thus, the set $T$ is large enough so that Breaker can indeed maintain an injective map $w\mapsto t_w$ for $w\in N_M(v)$. Furthermore, it is clear from the strategy description that a Maker/Breaker edge in $N_M(v)$ corresponds to a Maker/Breaker edge in the auxiliary game in $T$. Finally, since $\cS$ is a strategy for Breaker to prevent Maker in the auxiliary $b/2$-biased game to claim a copy of $K_{r-1}$ this implies that Breaker can indeed prevent Maker from claiming a copy of $K_{r-1}$ in $N_M(v)$, i.e.~the vertex $v$ is not in a copy of $K_r$ in Maker's graph.
\end{proof}

%%%%%%%%%%%%%%%%%%%%%%%%%%%%%%
%
%	DYNAMIC BECK
%
%%%%%%%%%%%%%%%%%%%%%%%%%%%%%%

\subsection{Proof of Theorem \ref{ourDynamicBreaker} (sketch)}
\label{proof:ourLemma5}

The following is a dynamic-board variant of \cite[Lemma 5]{bl2000}. We switch notation from $m$ to $p$ and from $b$ to $q$ for the bias of Maker and Breaker, respectively, to be consistent with the literature. 

\begin{lemma}\thlab{ourLemma5} In every dynamic-board $(\cH,p,q)$-game Breaker has a strategy such that at the end of the game at most $(1+q) f(\cH,p,q)$ edges of the hypergraph $\cH$ have all their vertices claimed by Maker, where $f(\cH,p,q)=\sum_{H\in E(\cH)}(1+q)^{-|H|/p}.$ 
\end{lemma}

The proof is a simple adaptation of the potential function technique as introduced by Erd\H{o}s and Selfridge~\cite{es1973} that was generalised by Beck~\cite{b1982} to biased Maker--Breaker games. We are unaware of such a dynamical-board variant thus the full proof follows. We follow notation and strategy of the proof of \cite[Theorem 20.1]{BeckBook}. 

\begin{proof}[Proof of \thref{ourLemma5}]
Let $\cH$, $p$, $q$ be as in the lemma and let $\mu$ be defined by $1+\mu = (1+q)^{1/p}$. 
Given two disjoint subsets $M$ and $B$ of the board $V=V(\cH)$ and an element $z\in V$ set  
\begin{align*}
\Phi(M,B) &= \sum_{H\in \cH: H\cap B =\emptyset} (1+\mu)^{-|H\sm M|},\text{ and}\\
\Phi(M,B,z) &= \sum_{z\in H\in \cH: H\cap B =\emptyset} (1+\mu)^{-|H\sm M|}
\end{align*}
and note straight away the following inequalities:
\begin{align}
\Phi(M\cup\{e\},B,z)&\le (1+\mu)\Phi(M,B,z),\lab{T1}\\
\Phi(M,B\cup\{e\},z)&\le \Phi(M,B,z).\lab{T2}.
\end{align}
For integers $r$ and $j$, let 
$b_r^{(j)}$ be the $j^{\mathrm{th}}$ element that Breaker picks in round $r$, and let $m_r^{(j)}$ be the $j^{\mathrm{th}}$ element that Maker picks in round $r$ if she decides to play according to Option (a) and pick elements in $V_{r-1}$ rather than enlarging $V_r$ (Option (b)). 
Furthermore, let $M_r$ and $B_r$ be the set of all elements of Maker and of Breaker, respectively, {\em after} round $r$, 
and let 
$M_{r,j}=M_r\cup \{m_{r+1}^{(1)},\ldots,m_{r+1}^{(j)}\}$ and 
$B_{r,j}=B_r\cup \{b_{r+1}^{(1)},\ldots,b_{r+1}^{(j)}\}$ (assuming that Maker/Breaker has claimed at least $j$ elements in round $r+1$). 

We now describe Breaker's strategy in round $r$. If there are less than $q$ unclaimed elements in $V_{r+1}$, then Breaker claims all of them. Otherwise, for every $1\leq j\le q$, sequentially, Breaker calculates $\Phi(M_r,B_{r-1,j-1},z)$ for every unclaimed element $z\in V_r\sm(M_r\cup B_{r-1,j-1})$ and claims the element $b_r^{(j)}$ which maximises this expression. Note that here we chose the element $b_r^{(j)}$ in $V_r$, and not in the whole board $V$. If, for some $j$, there are no unclaimed elements, then it is the end of Breaker's turn.

The crucial part of the potential function technique in positional games is to show that the {\em potential} $\Phi(M_{r+1},B_r)$ is decreasing (if evaluated after Makers move). But this is now straight-forward along the lines of the proof in \cite{BeckBook}. The only thing we have to notice is that in round $r+1$, if Maker chooses to claim elements in $V_r$, then their choices are on the same sub-board where Breaker claimed their elements in round $r$. 

\begin{claim}
For all $r\ge 1$, 
$\Phi(M_{r+1},B_r)\le \Phi(M_r,B_{r-1})$.
\end{claim}

\begin{proof} 
If Maker has played according to Option (b) in round $r+1$, then $\Phi(M_{r+1}, B_r) = \Phi(M_r, B_r) \le \Phi(M_r, B_{r-1})$, where the inequality follows from \eqref{T2}. Therefore, if Breaker was not able to claim $q$ elements in round $r$, then in round $r+1$ Maker is forced to play Option (b) and the claim follows. For the rest of the proof we can assume that Breaker is able to claim $q$ elements in round $r$. Without loss of generality, we can also assume that Maker claims $p$ elements in round $r+1$ (claiming fewer than $p$ elements only makes it easier for the desired inequality to hold). Let us denote these elements by $b_r^{(1)}, \ldots, b_r^{(q)}$ and $m_{r+1}^{(1)}, \ldots, m_{r+1}^{(p)}$, respectively.

We first note that $\Phi(M_r,B_{r-1,j+1})=\Phi(M_r,B_{r-1,j})-\Phi(M_r,B_{r-1,j},b_r^{(j+1)})$ for all $0\le j<q$, and 
 $\Phi(M_{r,j+1},B_{r})=\Phi(M_{r,j},B_{r})+\mu\Phi(M_{r,j},B_{r},m_{r+1}^{(j+1)})$ for all $0\le j<p$.  Hence, 
\begin{align*}
\Phi(M_{r+1},B_r)&= \Phi(M_r,B_{r-1}) 
	- \sum_{j=1}^q \Phi(M_r,B_{r-1,j-1},b_r^{(j)}) 
	+\mu \sum_{j=1}^p \Phi(M_{r,j-1},B_{r},m_{r+1}^{(j)}).
\end{align*}
% if Maker plays Option (a), and 
% $\Phi(M_{r+1},B_r)=\Phi(M_r,B_{r-1}) - \sum_{j=1}^q \Phi(M_r,B_{r-1,j-1},b_r^{(j)}) $ if Maker plays Option (b) in round $r+1$. In the latter case, this implies the claim since $\Phi(\cdot,\cdot,\cdot)$ is always non-negative. So we may assume that Maker plays Option (a). The remaining argument is then analogous to the argument in \cite{BeckBook} page 299f. 
% %
It suffices to show  
\begin{align}
\lab{aux343}
&\mu \sum_{j=1}^p \Phi(M_{r,j-1},B_{r},m_{r+1}^{(j)}) \le \sum_{j=1}^q \Phi(M_r,B_{r-1,j-1},b_r^{(j)}).
\end{align}
First note that for all $1\le j\le p$, 
$$\Phi(M_{r,j-1},B_{r},m_{r+1}^{(j)})\le (1+\mu)^{j-1}\Phi(M_{r},B_{r},m_{r+1}^{(j)})\le (1+\mu)^{j-1}\Phi(M_{r},B_{r-1,q-1},m_{r+1}^{(j)}),$$
where the first inequality follows from~\eqref{T1} and the second from~\eqref{T2}. 
Furthermore, the element $b_r^{(q)}$ is chosen by Breaker {\em before} Maker claims any of $m_{r+1}^{(j)}$  
and it is chosen to maximise the expression $\Phi(M_{r},B_{r-1,q-1},z)$ 
{\em over all unclaimed $z\in V_{r+1}$} (note that Maker sets $V_{r+1}$ after her move in round $r$). Hence, we deduce 
\begin{align*}%\lab{aux344}
\Phi(M_{r,j-1},B_{r},m_{r+1}^{(j)})\le (1+\mu)^{j-1}\Phi(M_{r},B_{r-1,q-1},b_r^{(q)}), 
\end{align*}
which readily implies that 
\begin{align}
\lab{aux346}
\mu \sum_{j=1}^p \Phi(M_{r,j-1},B_{r},m_{r+1}^{(j)}) \le \mu \Phi(M_{r},B_{r-1,q-1},b_r^{(q)}) \sum_{j=1}^p (1+\mu)^{j-1}.
\end{align}
To bound the right hand side of~\eqref{aux343} we note that for all $1\le j < q$, 
\begin{align*}%\lab{aux345}
&\Phi(M_{r},B_{r-1,j-1},b_r^{(j)})\ge \Phi(M_{r},B_{r-1,j-1},b_r^{(q)})\ge \Phi(M_{r},B_{r-1,q-1},b_r^{(q)}),
\end{align*}
where the first inequality follows again by $b_r^{(j)}$ maximising $\Phi(M_{r},B_{r-1,j-1},z)$ among all unclaimed elements $z$ in $V_{r+1}$ (and $b_r^{(j)}$ is chosen before $b_r^{(q)}$), and the second inequality follows from~\eqref{T2}. 
This implies that 
$$\sum_{j=1}^q \Phi(M_{r},B_{r-1,j-1},b_r^{(j)})\ge q  \Phi(M_r,B_{r-1,q-1},b_r^{(q)}).$$
This and~\eqref{aux346} imply that the inequality in~\eqref{aux343} follows from 
$$\mu \Phi(M_{r},B_{r-1,q-1},b_r^{(q)}) \sum_{j=1}^p (1+\mu)^{j-1}
\le q \Phi(M_r,B_{r-1,q-1},b_r^{(q)}),$$ 
which is true since $\Phi(M_{r},B_{r-1,q-1},b_r^{(q)})\ge 0$ and $(1 + \mu)^p - 1 = q$.
\end{proof} 
 
To finish the proof of  \thref{ourLemma5} note that 
$\Phi (\emptyset,\emptyset) = f(\cH,p,q)$ and that after Maker's move in the first round we have 
$\Phi(M_1,B_0) = \Phi(M_1,\emptyset) \le (1+\mu)^p f(\cH,p,q) = (1 + q) f(\cH, p ,q)$. 
By the above claim, the function $f(M,B)$ is decreasing if evaluated after Maker's move. Hence, at the end of the game, we have 
$\Phi(M,B)\le \Phi(M_1,B_0)\le (1+q) f(\cH,p,q)$, where now $M$ and $B$ denote the final set of elements that Maker and Breaker claimed, respectively. On the other hand, if Maker occupies all elements of an edge $H\in E(\cH)$ (at any point of the game), then the additive contribution to $\Phi(M,B)$ is 1 for each such set. 
Thus, if $(1+q)f(\cH,p,q)<k$ then Breaker has a strategy in the dynamic Maker--Breaker game so that at the end of the game Maker fully occupies fewer than $k$ sets $H\in E(\cH)$. 
\end{proof}

A complete proof of \thref{ourDynamicBreaker} including all details would involve repeating the proof of \cite[Theorem~1~(ii)]{bl2000} and replacing \cite[Lemma 5]{bl2000} by \thref{ourLemma5} wherever it occurs. Instead, we give an overview of the argument in~\cite{bl2000} and describe where the changes are needed. We leave out calculations that are identical. 

Without loss of generality we may assume that $H$ is such that $m_2(H)$ achieves maximum for $H' = H$, i.e.\ $m_2(H) = (e(H) - 1)/(v(H) - 2)$. If this is not the case, then Breaker can choose $H' \subseteq H$ which determines $m_2(H)$ (which might not be unique), and prevent Maker from creating $H'$. Call such an $H$ {\em $m_2$-maximal.} 

Bednarska and \L uczak first identify the dangerous structures that Maker can create during a $b$-biased $H$-game on $K_n$. We say that a graph $F$ is an $\bar H$-graph if it contains two vertices $v$, $w$ such that $F+vw$ is isomorphic to $H$. We shall write $F^{vw}$ to specify such vertices in $F$. Let $\cF=\{F_1^{v_1w_1},\ldots,F_t^{v_tw_t}\}$ be a family of different $\bar H$-graphs, whose vertex sets may intersect. Then $\cF$ is called a {\em $t$-fan} if $|\bigcap_{i=1}^t V(F_i^{v_iw_i})|\ge 2$, and it is called a {\em $t$-flower} if $|\bigcap_{i=1}^t V(F_i^{v_iw_i})|\ge 3$. Furthermore, a $t$-fan $\cF$ is called {\em simple} if $|\bigcap_{i=1}^t V(F_i^{v_iw_i})|=2$. If at some point during the game Maker's graph contains some $\bar H$-graph $F^{vw}$ such that the pair $\{v,w\}$ has not been claimed by Breaker, then $F^{vw}$ is {\em dangerous}. Similarly, a $t$-fan $\cF$ is called {\em dangerous} if its elements are dangerous subgraphs of Maker's graph. 

The following is the dynamic-board variant of~\cite[Lemma 9]{bl2000}.
\begin{lemma}\thlab{ourLemma9}
For every $m_2$-maximal graph $H$ that contains a cycle there exist positive constants $C$ and $n_0$ such that for every $n\ge n_0$ and $q\ge Cn^{1/m_2(H)}$ Breaker has a strategy such that at no stage of the dynamic $q$-biased $H$-game on $K_n$ Maker's graph contains a dangerous $q$-fan. 
\end{lemma}
Before sketching the proof of this lemma let us briefly explain how it implies \thref{ourDynamicBreaker}. 
When $H$ is a forest the argument is identical to the one given in~\cite{bl2000}: The requirement $b\ge 2n^{1/m_2(H)}$ implies that Breaker has a simple strategy in either case of~$H$ containing a path on three vertices or consisting of disjoint edges. 
When $H$ contains a cycle let $b\ge 2Cn^{1/m_2(H)}$ where $C$ is the constant given by \thref{ourLemma9}. Breaker's strategy in the dynamic $b$-based $H$-game played on $K_n$ is as follows. In every round, Breaker claims $b/2$ edges to follow the strategy of \thref{ourLemma9} and thus prevents Maker from building a dangerous $(b/2)$-fan. The remaining $b/2$ choices are used to block the pairs $v_1w_1,\ldots,v_tw_t$ of a dangerous $t$-fan $\{F_1^{v_1w_1},\ldots,F_t^{v_tw_t}\}$, if there is one for some $t\le b/2$. Note that this is indeed a winning strategy. Suppose that Maker can claim a copy of $H$, say in round $r$ of the game. Then Maker has created a dangerous $t$-fan for some $t\ge 1$ in round $r-1$ (or earlier). Since Maker plays with bias 1, she cannot create more than one dangerous $t$-fan per round. By \thref{ourLemma9}, $t\le b/2$, thus Breaker would block such a dangerous $t$-fan in round $r-1$. 

We now sketch the proof of \thref{ourLemma9}. The strategy for Breaker is based, yet again, on two strategies that are followed each with bias $q/2$. The following two lemmata encapsulate these strategies, they are the dynamic-board variants of Lemma~6 and Lemma~7 of~\cite{bl2000}, respectively. 
\begin{lemma}\thlab{ourLemma6}
For every $m_2$-maximal graph $H$ that contains a cycle there exist positive constants $C_1$, $n_1$ and $\delta <1$ such that for every $n\ge n_1$ and $q\ge C_1n^{1/m_2(H)}$ Breaker has a strategy such that at each moment of the dynamic $q$-biased $H$-game on $K_n$ there are no dangerous $s$-flowers in Maker's graph, where $s=q^{1-\delta}$. 
\end{lemma}
\begin{lemma}\thlab{ourLemma7}
For every $m_2$-maximal graph $H$ that contains a cycle and every positive $\delta <1$, there exist constants $C_2$ and $n_2$ such that for every $n\ge n_2$ and $q\ge C_2n^{1/m_2(H)}$ Breaker has a strategy in the dynamic $q$-biased $H$-game on $K_n$ which does not allow Maker to build $\frac12 \binom{q}{t}$ simple $t$-fans, where $t=q^{\delta/3}.$ 
\end{lemma}
To prove \thref{ourLemma9}, let $C_1$, $n_1$, $\delta$ be given as in \thref{ourLemma6}, let $C_2$ and $n_2$ be as in \thref{ourLemma7}, and set $C=2\max\{C_1,C_2\}$ and let $n$ be large enough. 
Let $\cS_1$ be the strategy given by \thref{ourLemma6} for the dynamic $(q/2)$-biased $H$-game on $K_n$, and let $\cS_2$ be the strategy given by \thref{ourLemma7} for the dynamic $(q/2)$-biased $H$-game on $K_n$. Then, using $q/2$ edges to follow strategy $\cS_1$ and $q/2$ edges to follow strategy $\cS_2$, Breaker can ensure that at no point during the dynamic $q$-biased $H$-game on $K_n$ Maker's graph contains a dangerous $s$-flower with $s= (q/2)^{1-\delta}$, nor $\frac12 \binom{q/2}{t}$ simple $t$-fans with $t=(q/2)^{\delta/3}.$ 
The reasoning why this implies that Maker's graph never contains a dangerous $q$-fan is exactly the same as in~\cite{bl2000}, see the last paragraph of the proof of Lemma~9 therein. 
 
Finally, we explain the proofs of \thref{ourLemma6,ourLemma7}. 

For \thref{ourLemma6}, define a {\em $t$-cluster} to be a graph consisting of $t$ copies of $H$ that intersect in at least 3 vertices (the difference to a $t$-flower is that here the building blocks are copies of $H$ rather than $\bar H$-graphs). Let $\cH_1$ be the collection of all edge sets on $K_n$ that form a copy of a {\em $t$-cluster}, where $t=t(H)$ is a large enough constant. 
Further, let $q_1 = q^{1-\delta_1/2}$ where $\delta_1=\delta_1(H)$ is a suitable constant as defined at the beginning of the proof of~\cite[Lemma 6]{bl2000}. The choice of parameters is identical to those in~\cite[Lemma 6]{bl2000}. Hence, calculations identical to those in~\cite[Lemma 6]{bl2000} show that $(q+1)f(\cH_1,1,q_1)<1$. It follows from \thref{ourLemma5} that Breaker has a strategy in the dynamic-board $(\cH_1,1,q)$-game to prevent Maker from claiming all edges of a $t$-cluster. 
We pause for a psychological note: This strategy does not forbid the creation of one copy of $H$ per se, which may seem contradictory to the final goal. For the strategy of \thref{ourLemma6} this is irrelevant though and shall not concern us. 
It then follows along the lines of the last two paragraphs of~\cite[Lemma 6]{bl2000} that by preventing such a $t$-cluster for sufficiently large $t$ (depending only on $H$), Breaker can also prevent that, in the dynamic $q$-biased $H$-game on $K_n$, Maker claims the edge set of a dangerous $s$-flower, for $s=q^{1-\delta}$, $\delta = \delta_1/4$, and $q\ge n^{1/m_2(H)}.$ 

For \thref{ourLemma7}, let $\cH_2$ consist of the edge sets of simple $t$-fans in $K_n$, for $t$ as in the lemma statement. Following the calculations in the proof of~\cite[Lemma 7]{bl2000} we find that $(q+1)f(\cH_2,1,q)<\frac12\binom{q}{t}$ if $q\ge C_2n^{1/m_2(H)}$ where $C_2$ and $n$ are large enough. It follows that Breaker has a strategy in the dynamic-board $(\cH_2,1,q)$-game such that at the end of the game at most $\frac12\binom{q}{t}$ of all $t$-fans are fully occupied by Maker.

\bibliographystyle{abbrv}
\bibliography{references}

\end{document}